\documentclass[11pt, oneside]{article}
\usepackage{amsfonts}
\usepackage{mathrsfs}
\usepackage{color}
\usepackage[colorlinks]{hyperref}
\usepackage{latexsym}
\usepackage{amssymb}
\usepackage{amsmath}
\usepackage{enumerate}
\usepackage{amsthm}
\usepackage{indentfirst}
\usepackage{mathtools}
\usepackage{tikz}
\usepackage{psfrag}
\usepackage{graphicx}
\usepackage{bm}
\usepackage{indentfirst}
\usepackage{esint}
\usepackage{color}
\usepackage{tikz}

\usepackage{pgfplots}
\usepackage{multirow}
\usepackage[rflt]{floatflt}
\usepackage{float}
\usepackage{authblk}
\usepackage{ulem}
\usepackage[square, comma, sort&compress, numbers]{natbib}
\usepackage{verbatim}
\numberwithin{equation}{section}
\allowdisplaybreaks

\newenvironment{proof2.1}{\medskip\noindent{\bf Proof of the Theorem 2.1:}\enspace}{\hfill \qed \newline \medskip}

\newenvironment{proof2.2}{\medskip\noindent{\bf Proof of the Theorem 2.2:}\enspace}{\hfill \qed \newline \medskip}

\newtheorem{theorem}{\color{black}\indent Theorem}[section]
\newtheorem{lemma}{\color{black}\indent Lemma}[section]

\newtheorem{definition}{\color{black}\indent Definition}[section]
\newtheorem{remark}{\color{black}\indent Remark}[section]

\newtheorem{example}{\color{black}\indent Example}[section]

\usepackage{amssymb,amsmath}
\begin{document}
\title{\LARGE\bf { Asymptotic Stability and the Forcing Term: An Analysis of Non-Newtonian Thin-Film Flows}
 }
\author[1]{Jinhong~Zhao}
\author[1]{Bin Guo\thanks{Corresponding author\newline \hspace*{6mm}{\it Email
addresses:} jhzhao23@mails.jlu.edu.cn~(Jinhong~Zhao),~bguo@jlu.edu.cn~(Bin Guo)}}

\affil[1]{School of Mathematics, Jilin University, Changchun, 
 Jilin Province 130012, China}

\renewcommand*{\Affilfont}{\small\it}
\date{} \maketitle
\vspace{-20pt}

{\bf Abstract:}\ We study a class of fourth-order quasilinear degenerate parabolic equations under both time-and space-dependent and time-and space-independent forces, modeling non-Newtonian thin-film flow over a solid surface in the "complete wetting" regime.  By analyzing the quantitative properties of solutions to non-autonomous differential inequalities and employing refined integral estimates, we derive two-sided convergence rate estimates for the solution. Numerical simulations are further provided to illustrate the consistency of our main results with the observed physical phenomena.



\section{Introduction}
\subsection{Problem}
In this paper, we investigate a class of fourth-order quasilinear degenerate parabolic equations with forcing terms:
\begin{equation}\label{1.1}
	\begin{cases}
		 u_t + a\left(u^{\alpha+2} |u_{xxx}|^{\alpha-1}u_{xxx} \right)_x = f(t,x),~~&(t,x) \in (0,\infty) \times\Omega,\\
		 u_x(t,x) =u_{xxx} (t,x) = 0,~~&(t,x) \in (0,\infty) \times  \partial \Omega,\\
		u(0,x)=u_{0}(x),~~&x \in \Omega,
	\end{cases}
\end{equation}
where 	$ \Omega=(0,L)$ is a bounded interval on the real line.    
The flow-behavior index $\alpha$ describes the rheological properties of the fluid:  $\alpha=1$ corresponds to a Newtonian fluid, while  $\alpha\neq 1$ corresponds to a non-Newtonian power-law fluid, with $\alpha>1$  representing shear-thinning fluid  (e.g., tomato ketchup), and  $\alpha<1$  representing shear-thickening fluid (e.g., cornstarch suspension).
The positive constant $a$ depends on the flow-behavior index, the surface tension coefficient, and the characteristic viscosity. 
$u(t,x)$ denotes the film thickness, and the forcing term $f(t,x)$ represents a volumetric source ($f>0$) or sink ($f<0$),  i.e., the volume of fluid injected or extracted per unit time per unit basal area.
Problem \eqref{1.1} describes the dynamical evolution of the thickness of a thin film with power-law rheology, driven by surface tension and subject to external injection or suction (see Figures \ref{tuu}). 
The Neumann-type boundary conditions $u_x=0$ and $u_{xxx}=0$ on $\partial \Omega $ reflect the zero-contact angle condition and the no-flux condition, respectively.

\begin{figure}[htbp]
	\centering
	\includegraphics[width=0.7\textwidth]{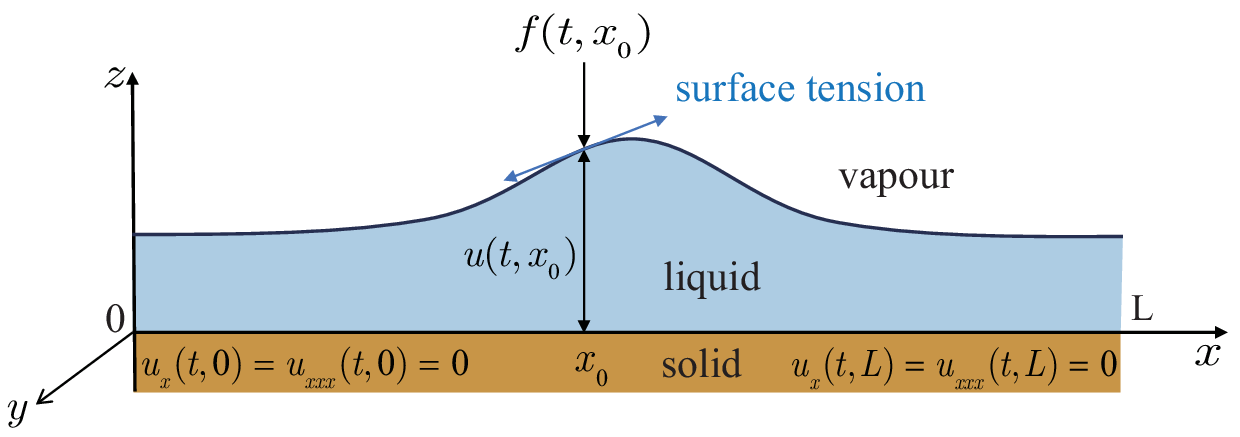}
	\caption{Cross-section of liquid film on impermeable solid bottom}\label{tuu}
\end{figure}


\subsection{Novelty and Significance}

To the best of the authors' knowledge, existing studies on the long-time behavior of non-Newtonian fluid films have mostly focused on autonomous systems—that is, systems where external mass exchange is neglected and film evolution is assumed to be driven solely by surface tension.  
In contrast, relatively little attention is given to the non-autonomous case. 

In practical settings, however, thin-film systems often undergo mass exchange with the external environment, such as through injection or evaporation. For example, lubricant injection through a porous substrate ($f(t,x) > 0$), thin-film growth in chemical vapor deposition ($f(t,x)>0$), 
and simulation of coating drying processes ($f(t,x)< 0$).  Thus, we ask some related questions:

$\bullet$ How is the dynamic behavior of the film governed by external injection or suction?

$\bullet$ Can a forcing term be designed to drive the film asymptotically to the desired state?

Before providing affirmative answers to the above questions (detailed in Sections 3–5), we first review previous studies and recent progress on thin-film problems. 

For $\alpha=1$ and $f=0$,  equation $\eqref{1.1}_1$ simplifies to the classical thin-film equation:
\begin{align}\label{1.3}
	u_t + \left(u^n u_{xxx} \right)_x = 0,
\end{align}
where $n$ corresponds to  distinct slip conditions on the liquid-solid interface:   "strong slippage" ($n \in (1, 2)$), "weak slippage" ($n \in (2, 3)$), "no-slip condition"($n = 3$) \cite{AOS}, "Navier-slip condition" ($n = 2 $) \cite{WJA}, and  Hele-Shaw flow ($n = 1$) \cite{LGF}.

The theoretical study of equation $\eqref{1.3}$ began with Bernis and Friedman \cite{FBAF}, who first proved the existence of global non-negative weak solutions in one-dimensional space, as well as the positivity and uniqueness for $n \ge 4$.
The key to this theoretical development lies in energy and entropy estimates.
Subsequently,  Beretta, Bertsch, and Dal Passo \cite{BBDP}, and Bertozzi and Pugh \cite{BP} (under periodic boundary conditions) further proved the existence and asymptotic behavior of one-dimensional non-negative weak solutions.  
Later, Dal Passo, Garcke and Gr\"{u}n \cite{DPGG} refined $\alpha-$entropy estimates used in \cite{CMEH} and \cite{GGD} and extended the results of either \cite{BBDP} or \cite{BP} to the higher-dimensional cases $N=2,3$. 
Their crucial element was the generalization of the formula of integration by parts
	$$\int_{0}^{L} f'(u)\cdot u_{x}^2  \cdot u_{xx} {\rm d}x = -\frac{1}{3} \int_{0}^{L} f''(u) \cdot u_x^4 {\rm d}x$$
for appropriately smooth functions with $u_{x}(0)=u_{x}(L)=0$ to higher dimensions, which will be essential in order to obtain entropy estimates. 
For further results in higher dimensions, interested readers may refer to \cite{L1,  L2}.

As for the convergence analysis, Carlen and Ulusoy \cite{CEU, EAS2} gave an explicit
 decay rate for the classical solution to equation \eqref{1.3} in the $H^1(\Omega)$-norm,  thereby improving the earlier results in the $L^1$ or $L^\infty$ norms obtained in \cite{BBDP, BP, CJT, LJS}.
 Furthermore, by means of standard regularized equations and tailored entropy functionals, Tudorascu \cite{ATL} demonstrated that the weak solution to equation \eqref{1.3} decays at an exponential rate within the $H^1(\Omega)$ norm. Very recently, Chugunova, Ruan, and Taranets \cite{CMRY} first investigated how the long-time behavior of solutions quantitatively depends on external forcing. Specifically, they analyzed the effect of time-independent, inhomogeneous forcing terms.

In contrast,  for the case $\alpha\neq1$, King \cite{KJR, KJRT} analyzed a doubly degenerate equation of the form
\begin{align}\label{1.4}
u_t + a\left(u^{\alpha+2} |u_{xxx}|^{\alpha-1}u_{xxx} \right)_x = 0,
\end{align}
which describes the spreading of power-law (Ostwald-de Waele) fluids.
A significant difference from \eqref{1.3} is that \eqref{1.4} lacks entropy estimates.  In this direction, for shear-thinning  power-law fluids  ($\alpha>1$),  Ansini and Giacomelli \cite{LALG} applied a two-step regularization scheme, Galerkin approximation, and energy methods to prove that  the solution of \eqref{1.4} converges to a steady state solution as time tends to infinity. This work was later improved and extended by Jansen, Lienstromberg, and Nik \cite{JJCLKN} who provided a systematic study of \eqref{1.4} for all  $\alpha>0$. Precisely, they established the global existence and long-time behavior of positive weak solutions as well as quantitative convergence rates. 
For more research on higher-order parabolic equations, some interesting references may be found in \cite{TPC,CLTPC,CLJJLV,LALGS,CLSM}.


Motivated by \cite{ATL,JJCLKN},  we analyze the asymptotic profile of solutions to \eqref{1.1} under the influence of forcing terms. Unlike the case $f(t,x)=0$, a nonzero $f(t,x)$  breaks the system's energy dissipation structure, leading to the loss of both $L^1$-mass conservation and the monotonicity of the energy functional. Consequently, standard techniques are not directly applicable, necessitating the development of new analytical tools. The main novelties of this work are as follows:

$\bullet$ The presence of time-and space-dependent forces  leads to non-autonomous differential inequalities through energy estimates. To conduct an asymptotic analysis of their solutions, we began with an ODE comparison principle (Lemma \ref{Lem01}), and then obtained two-sided estimates for the solution of the differential inequality (Lemma \ref{y1}). Consequently, this helps us to establish two‑sided estimates for the  convergence rate of the solution, which clearly demonstrates the substantial effect of the forcing terms on the convergence rate.  Briefly stated, our results can be summarized as follows:

 For $0<\alpha \le 1$, we established that the convergence of the solution to the mean value
$$\frac{\int_{\Omega}u_{0}(x)\, {\rm d}x+\int_{0}^\infty \int_{\Omega}  f(s,x) \, {\rm d}x   {\rm d}s}{|\Omega|}$$ is exponential when $f(t,x)$ decays exponentially, and polynomial when $f(t,x)$ decays polynomially. 
	This demonstrates that the asymptotic behavior of solutions is dominated by the forcing term.

However, for $\alpha>1$, we found that the solution always converges polynomially to $$\frac{\int_{\Omega}u_{0}(x)\, {\rm d}x+\int_{0}^\infty \int_{\Omega}  f(s,x) \, {\rm d}x   {\rm d}s}{|\Omega|},$$ regardless of whether $f(t,x)$ decays exponentially or polynomially.	Hence, the asymptotic behavior is governed by the dissipation term in this case.

$\bullet$ As a byproduct, we obtain the local $L^1$-in-time estimate for the dissipation functional $\int_{\Omega} u^{\alpha + 2} |u_{xxx}|^{\alpha+1} \, {\rm d}x$  as $t\to\infty$, and provide an explicit convergence rate. 

$\bullet$ Regarding autonomous systems,  \cite{JJCLKN} established explicit convergence rates for low initial energy solutions. However, in this paper, we develop unified methods to obtain analogous results for both low and high initial energy cases under the time-and space-independent force $f_0$. 	
In addition, we need to point out that the assumption \eqref{ff} on $f(t,x)$ in Theorem \ref{sth.long} is merely a technical condition to give a uniform positive lower bound for the solution; see Lemma \ref{lower.u}  for details. 
Particularly, it is not hard to verify that  the condition \eqref{ff} holds for the constant forcing term $f_0$.

\subsection{Main results of the paper}
In this paper, we obtain the global existence and long-time behavior of solutions to problem \eqref{1.1} under both time-and space-dependent and time-and space-independent forces at low initial energy.   To be more precise, our results are stated as follows.


For the  time-and space-dependent  force $f(t,x),$ we prove that the solution converges to $\frac{\int_{\Omega}u_{0}(x)\, {\rm d}x+\int_{0}^\infty \int_{\Omega}  f(s,x) \, {\rm d}x   {\rm d}s}{|\Omega|}$  in $H^1 (\Omega )$ and the local $L^1$-in-time estimate for the dissipation functional $\int_{\Omega} u^{\alpha + 2} |u_{xxx}|^{\alpha+1} \, {\rm d}x$  as $t\to\infty$, and provide an explicit convergence rate. 
Later,  we present numerical example (Example  \ref{ex1}), whose results are consistent with the theoretical analysis.


For the time-and space-independent force $f_0$, we prove that the solution coincides with $ \bar{u}_0 +tf_0$ in a finite time for $0<\alpha<1$, while it approaches this function at a polynomial rate for $\alpha>1$ and exponentially for $\alpha=1$.
Importantly, we observe that even at high initial energy, similar asymptotic behavior persists for sufficiently large $f_0$.


\subsection{Outline of the paper}

The structure of our paper is as follows:
In Section 2, we introduce some preliminary knowledge.  
Section 3 presents a key lemma concerning two-sided estimates for solutions to the differential inequality.
Next, in Sections 4 and 5, we investigate the global existence and  long-time behavior of positive weak solutions to the power-law thin-film problem \eqref{1.1} under both time-and space-dependent and time-and space-independent forces $f(t,x)$ and  $f_0$, respectively.
Numerical simulations of main results are presented in Section 6.  
We conclude in Section 7 with comments on the present study and an outlook for future research.


\section{Preliminaries}

In this section, we will give some notations and the related concepts.
In what follows, we denote by $\| \cdot \|_r (r \ge 1) $ the norm in $L^r(\Omega)$.
$C$ denotes a generic positive constant, which may differ at each appearance.  
For the sake of clarity, we omit the parameters $a$ in the subsequent analysis of problem \eqref{1.1}. 
In addition, we define  $\bar{u}_0=\frac{1}{|\Omega|}\int_\Omega u_0(x)\, {\rm  d}x$.

For $ k \in \mathbb{N} $ and $ p \in [1, \infty) $, we denote by $ W^{k,p}(\Omega) $ the usual Sobolev space with norm
$$  \| v \|_{W^{k,p}(\Omega)} = \left( \sum_{j=0}^{k} \| \partial^j v \|_p^p \right)^{\frac{1}{p}}.  $$
To account for the Neumann-type boundary conditions, we further introduce the Banach spaces  as follows:
\begin{align*}
W^{k,p}_{B}(\Omega) =
\begin{cases}
\{  v \in W^{k,p}(\Omega);v_x = v_{xxx} = 0 \text{ on } \partial \Omega \},
 &3+\frac{1}{p} <k\le 4 , \\
\{  v \in W^{k,p}(\Omega);v_x = 0 \text{ on } \partial \Omega \},
&1+\frac{1}{p} <k\le 3+\frac{1}{p}  , \\
W^{k,p}(\Omega),
&0 \le k\le 1+\frac{1}{p}.
\end{cases}
\end{align*}


\section{A key lemma: two-sided estimates}

In this section, we establish two-sided estimates for the solutions of inequality \eqref{05}.
We begin by introducing the following comparison principle. 
%
\begin{lemma}\label{Lem01}  Let $y(t)$ be the  non-negative solution of
	\begin{equation}\label{02}
	\begin{cases}
	y'(t)+\beta y^{\lambda}(t)= k(t)\geq0,\quad t>t_{0}\geq0,\\
	y(t_{0})=y_{t_0}>0.
	\end{cases}
	\end{equation}
	And
	$y_{1}(t)$ be the  solution of
	\begin{equation}\label{01}
	\begin{cases}
	y_{1}'(t)+\beta y_{1}^{\lambda}(t)=0,\quad t>t_{0}\geq0,\\
	y(t_{0})=y_{t_0}>0.
	\end{cases}
	\end{equation}	
	Then for any $t\geq t_{0}$,
	\begin{equation}\label{com}
	y(t)\geq y_{1}(t)=\begin{cases}
	y_{t_0}e^{-\beta (t-t_0)},  &\text{~if~}\lambda=1,\\
	y_{t_0}\big[1+(\lambda-1)\beta y_{t_0}^{\lambda-1}(t-t_{0})\big]^{\frac{1}{1-\lambda}},  ~&\text{~if~} \lambda \neq 1.
	\end{cases}
	\end{equation}
\end{lemma}

\begin{proof}
	First,   since the equation \eqref{01}$_1$ is separable, it is not hard to obtain the right side of  \eqref{com} by solving equation $y_1^{-\lambda}y'_1=-\beta.$
	Next, set $h(t)=y(t)-y_{1}(t),$ then $h(t_{0})=0.$ We claim that
	\begin{center}
		$h(t)\geq0,$  $~~~\forall t\geq t_{0}.$
	\end{center}
	If not, there exists $t_{1}>t_{0}$ satisfying $h(t_{1})<0$. Define $t_{2}=\sup\{t\in[0,t_{1}),h(t)=0\}.$ Hence
	\begin{center}
		$h(t_{2})=0$ and $h(t)<0,$  $~~~\forall t_{2}<t<t_{1}$.
	\end{center}
	Therefore, it is easy to prove that
	\begin{center}
		$h^{'}(t)=k(t)+\beta y_{1}^{\lambda}-\beta  y^{\lambda}\geq0$,  $~~~\forall t_{2}<t<t_{1},$
	\end{center}
	which implies
	$$h(t)\geq h(t_{2})=0.$$
	This is a contradiction. The proof is complete.
\end{proof}


Based on Lemma \ref{Lem01} stated above, we derive 
the following key lemma.
\begin{lemma}\label{y1}
	If $\beta>0$, $k(t) \geq0$ and $\int_{0}^{\infty}k(s)\,{\rm d}s < +\infty$,  then the non-negative solution $y(t)$ to
	\begin{equation}\label{05}
	\begin{cases}
	y^{'}(t)+\beta y^{\lambda}(t)\le k(t), \quad t> 0 ,\\
	y(0)=y_{0}>0,
	\end{cases}
	\end{equation}
	satisfies
	\begin{align*}
	y(t)\leq y_{0}+\int_{0}^{\infty}k(s)\,{\rm d}s := M_{0}<+\infty.
	\end{align*}
	Further, the solution $y(t)$ 
	satisfies the following inequality:
	
	{\rm(1)} If $\lambda>1$, then
	\begin{equation}\label{44}
		y_0\big[1+(\lambda-1)\beta y_0^{\lambda-1}t\big]^{\frac{1}{1-\lambda}} \le y(t)\leq
	M_{0}\left[ 1+\frac{\beta M_{0}^{\lambda-1}(\lambda-1)}{2}t\right] ^{\frac{1}{1-\lambda}}
	+\int_{\frac{t}{2}}^{t}k(s)\,{\rm d}s.
		\end{equation}
		
		{\rm(2)}	If $ 0<\lambda\le 1$, then
\begin{align}\label{44.1}
y_0e^{-\beta M_{0}^{\lambda-1} t}
\le y(t)
\leq M_{0}e^{\frac{-\beta M_{0}^{\lambda-1}t}{2}}
+ \int_{\frac{t}{2}}^{t}k(s)e^{-\frac{\beta M_{0}^{\lambda-1}(t-s)}{2}}\,{\rm d}s.
\end{align}	
\end{lemma}

\begin{proof}
First, integrating the equation $(\ref{05})_1$ from $0$ to $t$, we get
\begin{align*}
y(t)\le y_{0}+\int_0^{t}k(s)\,{\rm d}s-\beta\int_0^{t}y^{\lambda}(s)\,{\rm d}s
\leq y_{0}+\int_{0}^{\infty}k(s)\,{\rm d}s:=M_{0}<+\infty.
\end{align*}

Next, we prove \eqref{44} and \eqref{44.1}, there are two cases.
{\bf Case 1.} If $\lambda>1$, for any $t>0$, let $y_{2}(t)$ be the solution of
\begin{equation}\label{04}
\begin{cases}
y_{2}^{'}(s)+\beta y_{2}^{\lambda}(s)=0,\quad s\geq\frac{t}{2},\\
y_{2}\left( \frac{t}{2}\right) =y(\frac{t}{2}).
\end{cases}
\end{equation}
From $y^{'}(s)\le-\beta y^{\lambda}(s)+k(s)$ and Lemma \ref{Lem01}, then we have
\begin{align*}
y(t)-y(\frac{t}{2})&=\int_{\frac{t}{2}}^{t}y^{'}(s)\,{\rm d}s
\le-\beta \int_{\frac{t}{2}}^{t}y^{\lambda}(s)\,{\rm d}s+\int_{\frac{t}{2}}^{t}k(s)\,{\rm d}s\\
&\leq\int_{\frac{t}{2}}^{t}k(s)\,{\rm d}s-\beta \int_{\frac{t}{2}}^{t}y_{2}^{\lambda}(s)\,{\rm d}s
=\int_{\frac{t}{2}}^{t}k(s)\,{\rm d}s+\int_{\frac{t}{2}}^{t}y_{2}^{'}(s)\,{\rm d}s\\
&=y_{2}(t)-y_{2}(\frac{t}{2})+\int_{\frac{t}{2}}^{t}k(s)\,{\rm d}s,
\end{align*}
which implies
\begin{align}\label{y1.1}
	y(t)\leq y_{2}(t)+\int_{\frac{t}{2}}^{t}k(s)\,{\rm d}s.
\end{align}
Noting that
\begin{align}\label{y1.2}
y_{2}(t)&=y(\frac{t}{2})\left[ 1+\beta y^{\lambda-1}(\frac{t}{2})(\lambda-1)(t-\frac{t}{2})\right] ^{\frac{1}{1-\lambda}}\leq M_{0}\left[ 1+\beta M_{0}^{\lambda-1}(\lambda-1)\frac{t}{2}\right] ^{\frac{1}{1-\lambda}}.
\end{align}
Combining \eqref{y1.1}, \eqref{y1.2}, and Lemma \ref{Lem01}, we conclude
$$	y_0\big[1+(\lambda-1)\beta y_0^{\lambda-1}t\big]^{\frac{1}{1-\lambda}} \le y(t)\leq
M_{0}\left[ 1+\frac{\beta M_{0}^{\lambda-1}(\lambda-1)}{2}t\right] ^{\frac{1}{1-\lambda}}
+\int_{\frac{t}{2}}^{t}k(s)\,{\rm d}s.$$

{\bf Case 2.} If $0<\lambda\le 1$, then for any $t>0$, noticing that $y^{\lambda}(t)=y^{\lambda-1}(t)y(t)\geq M_{0}^{\lambda-1}y(t)$, we get
\begin{equation}\label{06}
\begin{cases}
y^{'}(t)+\beta M_{0}^{\lambda-1}y(t)\leq k(t),\\
y(\frac{t}{2})=y(\frac{t}{2})>0.
\end{cases}
\end{equation}
Since $(\ref{06})_1$ is a linear differential inequality, it follows from Lemma \ref{Lem01} that there exists an explicit formula for the solution
\begin{align*}
	y_0e^{-\beta M_{0}^{\lambda-1} t}
\le y(t)
\leq y(\frac{t}{2})e^{\frac{-\beta M_{0}^{\lambda-1}t}{2}}
+ \int_{\frac{t}{2}}^{t}k(s)e^{-\frac{\beta M_{0}^{\lambda-1}(t-s)}{2}}\,{\rm d}s.
\end{align*}
The proof of Lemma \ref{y1} is complete.

\end{proof}


\section{The time-and space-dependent force }

In this section, we first employ regularity theory for higher-order parabolic equations and energy methods to prove the local existence of weak solutions for all $\alpha>0$. Then, by establishing an uniform positive lower bound for  weak solutions, we extend solutions from local to global.
Finally, we develop a novel method to analyze the energy inequality, which further provides the long-time behavior of solutions.

\subsection{Local existence}

In this subsection, we give the definition and local existence of  positive weak solutions to problem \eqref{1.1}. 
The proof of this result is almost identical to that in \cite[Appendix A]{JJCLKN}, with the only modification being the proper handling of the time-and space-dependent force $f(t,x)$. 
Thus, we only state the main result herein.

\begin{definition}\label{sth.defi}

\rm For a given $T > 0$,  $u_0 \in H^1(\Omega)$ and $f(t,x) \in L^{\frac{\alpha + 1}{\alpha}}( (0, T); (W^{1,\alpha + 1}_{ B}(\Omega))') $.
	We say that a function
		$$
	u \in C\left([0, T]; H^1(\Omega)\right)  \cap L^{\alpha + 1}\left( (0, T); W^{3,\alpha + 1}_{  B}(\Omega)\right)  \text{~with~}   u_t \in L^{\frac{\alpha + 1}{\alpha}}\left( (0, T); (W^{1,\alpha + 1}_{ B}(\Omega))'\right)
	$$
		is a weak solution  to problem \eqref{1.1} in $[0,T] \times \Omega$  if $u$ satisfies
		
	{\rm (i)   } (Weak formulation) The function $u$ satisfies the differential equation $\eqref{1.1}_1$ in the weak sense, i.e.,
	\begin{align*}
	\int_0^T \int_\Omega u_t \varphi {\rm d}x {\rm d}t = \int_0^T \int_\Omega u^{\alpha + 2}  |u_{xxx}|^{\alpha - 1} u_{xxx}  \varphi_x \, {\rm d}x  {\rm d}t
	+\int_0^T \int_\Omega f(t,x) \varphi \, {\rm d}x {\rm d}t  ,
	\end{align*}
		for all test functions $\varphi \in L^{\alpha + 1}( (0, T); W^{1,\alpha + 1}_{ B}(\Omega))  $.
		
	{\rm (ii)}   (Initial and boundary values) The function $u$ satisfies the contact angle condition $u_x = 0$ on $\partial \Omega$ and the initial condition $\eqref{1.1}_3$ pointwise.   	
	%


\end{definition}

\begin{theorem}\label{sth.local}
	 Let $\alpha>0$.  If initial value $u_0 \in H^1(\Omega)$ satisfies $u_0(x) > 0$, for all $x \in \bar{\Omega}$, and
	 $f(t,x) \in L^{\frac{\alpha+1}{\alpha}}  ((0,\infty)  \times \Omega) \cap L^1( (0, \infty); H^1(\Omega))$,
	 then there exists a positive $T > 0$
	  and at least one positive weak solution of   problem \eqref{1.1}
	$$
	u \in C\left([0, T]; H^1(\Omega)\right)  \cap L^{\alpha + 1}\left( (0, T); W^{3,\alpha + 1}_{ B}(\Omega)\right)  \text{~with~}   u_t \in L^{\frac{\alpha + 1}{\alpha}}\left( (0, T); (W^{1,\alpha + 1}_{B}(\Omega))'\right)
	$$
	 on $(0, T)$ in the sense of Definition {\rm \ref{sth.defi}}. Moreover, such a solution has the following properties:
	
		{\rm (i)  }  (Positivity)
		$u$ is bounded away from zero:
		\begin{align}\label{sth.local.posi}
				  u(t,x) >0, \quad (t,x) \in [0,T]  \times \bar{\Omega}.
		\end{align}
	
		{\rm (ii) } (Mass equation)
		$u$ satisfies the mass equation  in the sense that
		\begin{align}\label{sth.local.mass}
		\int_{\Omega} u(t,x)   \, {\rm d}x = \int_{\Omega} u_0(x) \, {\rm d}x  +\int_{0}^t \int_{\Omega}  f(s,x) \, {\rm d}x   {\rm d}s, \quad t \in [0,T].
		\end{align}

		{\rm (iii)} (Energy equation)
		$u$ satisfies the energy equation  in the sense that
		\begin{align}\label{sth.local.energy}
		E[u](t) + \int_0^t D[u](s) \, {\rm d}s = E[u_0]+
		\int_{0}^t \int_{\Omega}  f_x(s,x) u_x \, {\rm d}x   {\rm d}s,
		\end{align}
		for almost every $t \in [0, T]$, where \textit{energy functional} $E[u](t)$ and \textit{dissipation functional} $D[u](t)$
		are defined by
		$$E[u](t) = \frac{1}{2} \int_{\Omega} |u_x|^2 \, {\rm d}x ~~~\text{and}~~~D[u](t)=\int_{\Omega} u^{\alpha + 2} |u_{xxx}|^{\alpha+1} \, {\rm d}x.$$
\end{theorem}


\begin{remark}\label{global}
	{\rm    The proof of Theorem \ref{sth.local} shows that the local existence time $T$ depends only  on the positive lower bound of the initial data. 
		By restarting the system with the final state $ u(T) $ as the new initial data, the positive weak solution  $u$ to problem \eqref{1.1}  can be extended beyond time $ T $. 
		Since $u(T,x)>0$ in $\Omega$, the solution can be continued until a maximal time $ T^* > 0 $, when $ u(T^*, x) = 0 $ for some $ x \in \bar{\Omega} $. }
		%
\end{remark}

\subsection{Long-time behavior}
In this subsection, we first state our main results concerning the global existence and long-time behavior of weak solutions, and provide explicit two-sided estimates for the convergence rate.

\begin{theorem}\label{sth.long}
	Fix $\alpha > 0$. If $0<u_0(x)\in H^1(\Omega)$, $f(t,x)\in L^{\frac{\alpha + 1}{\alpha}}( (0, \infty) \times \Omega) \cap L^1( (0, \infty)\times  \Omega)$ satisfy
	\begin{align}\label{ff}
		\|u_{0,x}\|_2 < \left| \Omega\right|^{-\frac{1}{2}}\bar{u}_0  \text{~~~and~~~}
		\|f_x(t,x)\|_2  \le  \left| \Omega\right|^{-\frac{3}{2}}
		\int_{\Omega}  f(t,x) \, {\rm d}x,~~\forall t>0  .
	\end{align}
	Then  problem \eqref{1.1} possesses at least one global positive weak solution
	\begin{align*}
		&u \in C\left( [0,\infty); H^1(\Omega)\right)  \cap L^{\alpha+1 }\left( (0,\infty); W^{3,\alpha + 1}_{B}(\Omega) \right)  \text{~with~}\nonumber\\
		&u_t \in L^{\frac{\alpha + 1}{\alpha}}\left( (0,\infty); (W^{1,\alpha + 1}_{B}(\Omega))'\right) ,
	\end{align*}
	satisfying the boundary condition $u_x = 0$ on $\partial\Omega$ pointwise for almost every  $t \ge 0$.
	Moreover, this global solution has the following asymptotic behavior:

	{\rm (1)} In the shear-thinning case $1 < \alpha < \infty$, there exists a constant $C > 0$ such that for all $t >0$,
	\begin{align}\label{con2}
\frac{\|u_{0,x}\|_2}{\left[ 1+C\|u_{0,x}\|_2^{\alpha-1}t \right] ^{\frac{1}{\alpha-1}}}& \le 	\left\| u - \bar{u}_0 - \frac{\int_{0}^\infty \int_{\Omega}  f(s,x) \, {\rm d}x   {\rm d}s}{|\Omega|} \right\|_{H^1(\Omega)}\nonumber\\
	&\leq \frac{C M_0}{\left[ 1+CM_0^{\alpha-1}t \right] ^{\frac{1}{\alpha-1}}}
	+ C \int_{\frac{t}{2}}^\infty \int_{\Omega}  f(t,x) \, {\rm d}x {\rm d}s .
	\end{align}
	
	{\rm (2)}  In the shear-thickening case $0 < \alpha < 1$ or in the Newtonian case $\alpha = 1$, there exists a constant $C > 0$ such that  for all $t >0$,
		\begin{align}\label{con21}
		\|u_{0,x}\|_2e^{-C M_{0}^{\alpha-1} t}
		&\le \left\| u - \bar{u}_0 - \frac{\int_{0}^\infty \int_{\Omega}  f(s,x) \, {\rm d}x   {\rm d}s}{|\Omega|} \right\|_{H^1(\Omega)}\nonumber\\
		&\leq  CM_0e^{-CM_0^{\alpha-1}t}
		+ C \int_{\frac{t}{2}}^\infty  \int_{\Omega}  f(t,x) \, {\rm d}x {\rm d}s,
		\end{align}	
where $0<M_0:=\|u_{0,x}\|_2 + \int_0^\infty \|f_x(s)\|_2\, {\rm d}s <+\infty$.
\end{theorem}

\begin{remark}\label{u0}
{\rm  	
	Consider  $u_0(x)=A_1+B_1\cos(\frac{\pi x}{m_1})$, $f(t,x) = (1+t)^{-\beta}(A_2+B_2\cos(\frac{\pi x}{m_2}) )$  with $\beta>1$ or $f(t,x) = e^{-(\beta-1)t}(A_2+B_2\cos(\frac{\pi x}{m_2}) ) $  with $\beta>1$  in $\Omega  =[0,L]$.  If the parameters satisfy $k_i|B_i|\pi<\sqrt{2}A_i$ and $k_im_i=L$ for $i=1,2$, then all the assumptions of Theorem \ref{sth.long} are satisfied. }
\end{remark}

\begin{remark}\label{u000}
	{\rm 	Substituting the examples in Remark \ref{u0} into Theorem \ref{sth.long},  a straightforward calculation shows that

	(1) For $0<\alpha\le 1$, if  $f(t,x) = (1+t)^{-\beta}(A+B\cos(\frac{\pi x}{m}) ) ~(\beta>1)$, then
			$$	\|u_{0,x}\|_2e^{-C M_{0}^{\alpha-1} t}
			\le \left\| u - \bar{u}_0 - \frac{1}{\beta-1} \right\|_{H^1(\Omega)}
			\leq  CM_0e^{-CM_0^{\alpha-1}t}
			+ \frac{C}{\beta-1} \left(
			1+\frac{t}{2}\right) ^{-\beta+1}.$$

			If   $f(t,x) = e^{-(\beta-1)t}(A+B\cos(\frac{\pi x}{m}) ) ~(\beta>1)$,	then
			$$	\|u_{0,x}\|_2e^{-C M_{0}^{\alpha-1} t}
			\le \left\| u - \bar{u}_0 - \frac{1}{\beta-1} \right\|_{H^1(\Omega)}
			\leq  CM_0e^{-CM_0^{\alpha-1}t}
			+ Ce^{-\frac{\beta-1}{2}t}.$$

(2) For $\alpha> 1$, if   $f(t,x) = (1+t)^{-\beta}(A+B\cos(\frac{\pi x}{m}) ) ~(\beta>1)$, then
	$$\frac{\|u_{0,x}\|_2}{\left[ 1+C\|u_{0,x}\|_2^{\alpha-1}t \right] ^{\frac{1}{\alpha-1}}} \le 	\left\| u - \bar{u}_0 - \frac{1}{\beta-1} \right\|_{H^1(\Omega)}
	\leq \frac{C M_0}{\left[ 1+CM_0^{\alpha-1}t \right] ^{\frac{1}{\alpha-1}}}
	+\frac{C}{\beta-1} \left( 1+\frac{t}{2}\right) ^{-\beta+1}.$$
	
	If $f(t,x) = e^{-(\beta-1)t}(A+B\cos(\frac{\pi x}{m}) )~(\beta>1) $, then
	$$	\frac{\|u_{0,x}\|_2}{\left[ 1+C\|u_{0,x}\|_2^{\alpha-1}t \right] ^{\frac{1}{\alpha-1}}} \le \left\| u - \bar{u}_0 - \frac{1}{\beta-1} \right\|_{H^1(\Omega)}
	\leq \frac{C M_0}{\left[ 1+CM_0^{\alpha-1}t \right] ^{\frac{1}{\alpha-1}}}
	+  C e^{-\frac{\beta-1}{2}t}.$$

From the above results, it can be seen that for $0<\alpha<1$, if the time-and space-dependent force $f(t,x)$ decays exponentially in time, then the solution converges exponentially to   $\bar{u}_0 + \frac{1}{\beta-1}$;
%
if $f(t,x)$ decays only polynomially, the convergence remains polynomial, as opposed to exponentially. 
For $\alpha>1$, even if $f(t,x)$ decays exponentially in time, the solution still converges polynomially to   $\bar{u}_0 + \frac{1}{\beta-1}$.
%
	These demonstrate that the forcing term has a significant impact on the convergence rate of the solution.

}
\end{remark}




%
To prove Theorem \ref{sth.long}, we give several key lemmas. First, we present a positive lower bound for the solution to problem  \eqref{1.1}.
	\begin{lemma}\label{lower.u}
		If all the conditions of Theorem \ref{sth.long} are satisfied,  then  there exists a constant $m:=\bar{u}_0 - \left| \Omega\right|^{\frac{1}{2}} \|u_{0,x}\|_2>0$  such that
		$$u(t,x)\geq  m,~~\forall t\ge 0.$$
	\end{lemma}

\begin{proof}
First,  from \eqref{sth.local.energy} and H${\rm \ddot{o}}$lder's inequality,  we obtain
\begin{align}\label{ddt1}
\frac{1}{2}\frac{\rm d}{{\rm d}t} \|u_x\|_2^2 + \int_{\Omega} u^{\alpha + 2} |u_{xxx}|^{\alpha+1} \, {\rm d}x
=\int_{\Omega} f_x(t,x)u_{x} \, {\rm d}x
\le \|f_x(t,x)\|_2 \|u_{x}\|_2,~~~\forall t\in[0,T],
\end{align}
which implies
\begin{align}\label{ux}
\|u_x\|_2 \le \|u_{0,x}\|_2 + \int_0^t  \|f_x(s,x)\|_2 \,{\rm d}s,~~~\forall t\in[0,T],
\end{align}
where $T$ is the existence time of the solution in Theorem \ref{sth.local}.

In  addition, according to \eqref{sth.local.mass}, we get
\begin{align}\label{sth.local.mass1}
\int_{\Omega} \left(u- \bar{u}_0 -  \frac{1}{|\Omega|} \int_{0}^t \int_{\Omega}  f(s,x) \, {\rm d}x  {\rm d}s \right)   \,{\rm d}x =0,~~~\forall t\in[0,T].
\end{align}
Then by \eqref{sth.local.mass1}, the fundamental theorem of calculus and  Poincar\'e's  inequality, we find
\begin{align}\label{u-u}
\left|   u - \bar{u}_0 - \frac{1}{|\Omega|} \int_{0}^t \int_{\Omega}  f(s,x) \, {\rm d}x  {\rm d}s  \right|
\le \left| \Omega\right|^{\frac{1}{2}}  \| u_x \|_{2},
\end{align}
and
\begin{align}\label{u-u1}
\int\limits_{\Omega} \left(u - \bar{u}_0 - \frac{1}{|\Omega|} \int_{0}^t \int_{\Omega}  f(s,x) \, {\rm d}x  {\rm d}s  \right)^{2} {\rm d}x
\leq \left( \frac{|\Omega|}{\pi}\right) ^2\|u_x\|_2^2,~~~\forall t\in[0,T].
\end{align}
Combining \eqref{ux} with \eqref{u-u}, we get
\begin{align}\label{123}
\left|  u - \bar{u}_0 - \frac{1}{|\Omega|} \int_{0}^t \int_{\Omega}  f(s,x) \, {\rm d}x  {\rm d}s \right|
\leq \left| \Omega\right|^{\frac{1}{2}}  \left(\|u_{0,x}\|_2 + \int_0^t  \|f_x(s,x)\|_2 \,{\rm d}s \right),~~~\forall t\in[0,T].
\end{align}
Since $\|u_{0,x}\|_2 < \left| \Omega\right|^{-\frac{1}{2}}\bar{u}_0 $ and $ \|f_x(t,x)\|_2 \le  \left| \Omega\right|^{-\frac{3}{2}}
\int_{\Omega}  f(t,x) \, {\rm d}x $ by assumptions,  from \eqref{123}, we deduce  
\begin{align}\label{po1}
u(t,x)\ge  m ,~~~\forall t\in[0,T].
\end{align}
Therefore, solution $ u $ to \eqref{1.1} on $ [0, T] $ remain strictly bounded away from zero. By bootstrapping as in Remark \ref{global},
this solution can be extended globally for all $t \ge 0$, yielding a positive weak solution
$  u \in C([0, \infty); H^1(\Omega)) \cap L^{\alpha + 1}((0, \infty); W^{3,\alpha + 1}_{B}(\Omega))$
that satisfies
$$u(t,x)\geq  m >0,~~\forall t\in[0,\infty).$$
	This concludes the proof.  	
\end{proof}

\begin{remark}
{\rm 
As revealed by Lemma \ref{lower.u}, the conditions $\|u_{0,x}\|_2 < \left| \Omega\right|^{-\frac{1}{2}}\bar{u}_0 $ and  $ \|f_x(t,x)\|_2 \le  \left| \Omega\right|^{-\frac{3}{2}}
		\int_{\Omega}  f(t,x) \, {\rm d}x $ in Theorem \ref{sth.long} serve to guarantee that the solution is  bounded from below by a positive universal constant.}
\end{remark}


Based on the estimate of lower bound for the solution, we can now derive the following differential inequality for the energy functional $E[u]$,  which plays a crucial role in studying the long-time behavior of solutions to \eqref{1.1}.


\begin{lemma}\label{dissipation}
For $\alpha > 0$ and $u_0 \in H^1(\Omega)$ satisfying  $u_0(x) > 0$ for all $x \in \bar{\Omega}$. Let $u$ be a weak solution to \eqref{1.1}. 
Then for almost every $t \ge 0$, the following inequality holds
	\begin{align}\label{dissipation.11}
	\frac{\rm d}{{\rm d}t} E[u](t) + m_1\left( E[u](t)\right) ^{\frac{\alpha+1}{2}}
	\le \sqrt{2} \|f_x(t,x)\|_2\left(  E[u](t)\right)^{\frac{1}{2}} ,  
	\end{align}
		where  $$m_1:=2^{\frac{\alpha+1}{2}}|\Omega|^{-\frac{5\alpha + 3}{2}} m^{\alpha + 2}.$$
\end{lemma}

\begin{proof}
	First, we claim that the following inequality holds
	\begin{align}\label{zzz3}
	|u_x| \leq \int_{\Omega} |u_{xx}| \, {\rm d}x \leq |\Omega| \int_{\Omega} |u_{xxx}| \, {\rm d}x .
	\end{align}
	Indeed,  by the fundamental theorem of calculus and boundary condition $u_x|_{\partial \Omega}=0$, then there exists $y_0\in \partial \Omega$ satisfying
	\begin{align}\label{zzz1}
		|u_x| = |u_x(x)-u_x(y_0)| = \left| \int_{y_0}^x u_{xx}(y) \, {\rm d}y \right|
		\le \int_{\Omega} |u_{xx}| \, {\rm d}x,~~~\forall x\in\Omega.
	\end{align}
	Next, observe that there exists  $z_0\in \Omega$ satisfying $u_{xx}(z_0)\le 0.$
	If not, then for all $z_0\in \Omega,$ we have $u_{xx}(z_0)> 0, $  which contradicts the boundary condition  $u_x|_{\partial \Omega}=0$.
Then, we deduce
\begin{align}\label{zzz2}
	|u_{xx}|\le  |u_{xx}(x)-u_{xx}(z_0)|= \left| \int_{z_0}^{x} u_{xxx}(z)\,{\rm d}z\right|  \le \int_\Omega |u_{xxx}| \, {\rm d}x,~~~\forall x\in\Omega.
\end{align}
Substituting \eqref{zzz2} into \eqref{zzz1} yields \eqref{zzz3}.	
	Then applying  H${\rm \ddot{o}}$lder's inequality and the definition of $E[u](t)$,  \eqref{zzz3} implies
	\begin{align}\label{dissipation.2}
		E[u](t) \le \frac{1}{2}  |\Omega|^{\frac{5\alpha + 3}{\alpha + 1}}\|u_{xxx}\|_{\alpha + 1}^2.
	\end{align}	
	Subsequently, Lemma \ref{lower.u} and \eqref{dissipation.2} yields
	\begin{align}\label{dissipation.3}
		D[u](t) = \int_{\Omega}  |u|^{\alpha + 2} |u_{xxx}|^{\alpha + 1} \, {\rm d}x
		\geq m^{\alpha+2}  \|u_{xxx}\|^{\alpha + 1}_{\alpha + 1}
		\geq  m_1 \left( E[u](t)\right) ^{\frac{\alpha + 1}{2}},
	\end{align}
	where $m_1= 2^{\frac{\alpha+1}{2}}|\Omega|^{-\frac{5\alpha + 3}{2}} m^{\alpha + 2}$ and $m$ is gived by Lemma \ref{lower.u}.
Next, taking into account \eqref{dissipation.3} into \eqref{sth.local.energy} and using H${\rm \ddot{o}}$lder's  inequality, we further obtain
 \begin{align*}
	\frac{\rm d}{{\rm d}t} E[u](t)
	&= -D[u](t) + \int_{\Omega} f_x(t,x)u_{x} \, {\rm d}x  \leq -m_1\left( E[u](t)\right) ^{\frac{\alpha+1}{2}}
	+ \|f_x(t,x)\|_2 \|u_{x}\|_2\nonumber\\
	& =-m_1\left( E[u](t)\right) ^{\frac{\alpha+1}{2}}
	+ \sqrt{2}  \|f_x(t,x)\|_2 \left(E[u](t) \right) ^{\frac{1}{2}} ,  	
	\end{align*}	
	for almost every $t \ge 0$.
	This concludes the proof.  	
\end{proof}


~

{\bf Proof of Theorem \ref{sth.long}.}
First, by Lemma \ref{dissipation}, and multiplying both sides of \eqref{dissipation.11} by $\left(E[u](t) \right) ^{-\frac{1}{2}}$, we deduce
\begin{align}\label{dissipation.zz}
	\frac{\rm d}{{\rm d}t}\|u_x\|_2 + m_1 2^{-\frac{\alpha+1}{2}}\|u_x\|_2 ^{\alpha}
	\le \|f_x(t,x)\|_2 ,~~~\forall t > 0.
\end{align}
Below we analyze the long-time behavior of solutions, considering two cases.

{\bf Case 1.}  If  $\alpha>1,$  then  from \eqref{dissipation.zz} and Lemma \ref{y1}, we have
\begin{align}\label{E.1}
		\|u_{0,x}\|_2\big[1+&(\alpha-1)m_1 2^{-\frac{\alpha+1}{2}} \|u_{0,x}\|_2^{\alpha-1}t\big]^{\frac{1}{1-\alpha}} \le  \|u_x\|_2\nonumber\\
	&\le  M_0\left[1+2^{-\frac{\alpha+3}{2}}m_1(\alpha-1) M_0^{\alpha-1}t \right] ^{\frac{1}{1-\alpha}}
	+ \int_{\frac{t}{2}}^t \|f_x(s,x)\|_2\, {\rm d}s,~~~\forall t>0.
\end{align}
Provided that
\begin{align}\label{M0}
	0<M_0:=\|u_{0,x}\|_2 + \int_0^\infty \|f_x(s,x)\|_2\, {\rm d}s <+\infty.
\end{align}
Thus, from \eqref{u-u1} and \eqref{E.1}, we have

\begin{align}\label{H1.1}
\|u_x\|_2  
\le 	&\left\| u - \bar{u}_0 - \frac{1}{|\Omega|} \int_{0}^\infty \int_{\Omega}  f(s,x) \, {\rm d}x  {\rm d}s\right\|_{H^1(\Omega)} \nonumber\\
\le & \left\| u - \bar{u}_0 - \frac{1}{|\Omega|} \int_{0}^t \int_{\Omega}  f(s,x) \, {\rm d}x  {\rm d}s\right\|_{H^1(\Omega)} +\left\|\frac{1}{|\Omega|} \int_{t}^\infty \int_{\Omega}  f(s,x) \, {\rm d}x  {\rm d}s  \right\|_{H^1(\Omega)}\notag\\
	\le &\left\| u - \bar{u}_0 - \frac{1}{|\Omega|} \int_{0}^t \int_{\Omega}  f(s,x) \, {\rm d}x  {\rm d}s  \right\|_2
	+ \|u_x\|_2 +  \left\|\frac{1}{|\Omega|} \int_{t}^\infty \int_{\Omega}  f(s,x) \, {\rm d}x  {\rm d}s  \right\|_2 \notag\\
	\leq & \; C_1 M_0\left[1+2^{-\frac{\alpha+3}{2}}m_1(\alpha-1) M_0^{\alpha-1}t \right] ^{\frac{1}{1-\alpha}}
	+ C_1 \int_{\frac{t}{2}}^t \|f_x(s,x)\|_2\, {\rm d}s \notag\\
	&+ |\Omega|^{-\frac{1}{2}}\int_{t}^\infty \int_{\Omega}  f(s,x) \, {\rm d}x  {\rm d}s ,  
\end{align}
where $C_1= \frac{|\Omega|}{\pi}+1$. Combining \eqref{E.1} and \eqref{H1.1} yields \eqref{con2}.

{\bf Case 2.}  If $0<\alpha\le 1,$  then from \eqref{dissipation.zz} and Lemma \ref{y1}, we obtain
\begin{align}\label{E.2}
\|u_{0,x}\|_2&e^{-m_1 2^{-\frac{\alpha+1}{2}} M_{0}^{\alpha-1} t}
\le  \|u_x\|_2\nonumber\\
&\le M_0 e^{-2^{-\frac{\alpha+3}{2}}m_1M_0^{\alpha-1}t}
+ \int_{\frac{t}{2}}^t  e^{-2^{-\frac{\alpha+3}{2}}m_1M_0^{\alpha-1}(t-s)}\|f_x(s,x)\|_2\, {\rm d}s,~~~\forall t>0.
\end{align}
Provided that
$$0<M_0:=\|u_{0,x}\|_2 + \int_0^\infty \|f_x(s,x)\|_2\, {\rm d}s <+\infty.$$
Thus, similar to estimate  \eqref{H1.1}, we get \eqref{con21}. 
The proof of Theorem \ref{sth.long} is complete.
\(\square\)

As a byproduct, we obtain the following local $L^1$-in-time decay estimate for the dissipation functional.
\begin{theorem}\label{D}
		If all the conditions of Theorem \ref{sth.long} are satisfied, and $u$ be a positive weak solution. Then there exists a constant $C>0$, independent of $t$, such that the dissipation functional $D[u]$ satisfies the following estimate 
	\begin{align}
	\int_{\frac{t}{2}}^{t}D[u](s){\rm d} s\leq \frac{C}{t}\int_{\frac{t}{4}}^{\frac{t}{2}}E[u](s){\rm d} s+C\left( \int_{\frac{t}{4}}^{t} \|f_x(s,x)\|_2^2 {\rm d} s\right)^{\frac{1}{2}}\left( \int_{\frac{t}{4}}^{t}E[u](s) {\rm d} s\right)^{\frac{1}{2}}.
	\end{align}
	More specifically, 
	
	{\rm (1)} In the shear-thinning case $1 < \alpha < \infty$, the following inequality holds
	\begin{align}
	\int_{\frac{t}{2}}^{t}D[u](s){\rm d} s
	\leq & \; CM_0^2\left(1+ CM_0^{\alpha-1}t\right)^{\frac{2}{1-\alpha}} 
	+ C\left( \int_{\frac{t}{2}}^t \|f_x(s,x)\|_2\, {\rm d}s\right)^2 \nonumber\\
	&+ CM_0\left(1+ CM_0^{\alpha-1}t\right)^{\frac{3-\alpha}{1-\alpha}} \left( \int_{\frac{t}{4}}^t \|f_x(s,x)\|_2^2\, {\rm d}s\right)^{\frac{1}{2}}\nonumber\\
	&	+  C\left( \int_{\frac{t}{4}}^t \|f_x(s,x)\|_2^2\, {\rm d}s\right)^{\frac{1}{2}} \left( \int_{\frac{t}{4}}^t\left( \int_{\frac{s}{2}}^s \|f_x(\tau,x)\|_2\, {\rm d}\tau\right)^2 \, {\rm d}s\right)^{\frac{1}{2}}.
	\end{align}
	
	{\rm (2)}  In the shear-thickening case $0 < \alpha < 1$ or in the Newtonian case $\alpha = 1$, the following inequality holds
	\begin{align}
	\int_{\frac{t}{2}}^{t}D[u](s){\rm d} s
	\leq & \; CM_0^2e^{-CM_0^{\alpha-1}t}
	+ C\left( \int_{\frac{t}{2}}^t \|f_x(s,x)\|_2 e^{-CM_0^{\alpha-1}(t-s)}\, {\rm d}s\right)^2 \nonumber\\
	&+ CM_0e^{-CM_0^{\alpha-1}t} \left( \int_{\frac{t}{4}}^t \|f_x(s,x)\|_2^2\, {\rm d}s\right)^{\frac{1}{2}}	\nonumber\\
	&+  C\left( \int_{\frac{t}{4}}^t \|f_x(s,x)\|_2^2\, {\rm d}s\right)^{\frac{1}{2}} \left( \int_{\frac{t}{4}}^t\left( \int_{\frac{s}{2}}^s \|f_x(\tau,x)\|_2 e^{-CM_0^{\alpha-1}(s-\tau)}\, {\rm d}\tau\right)^2  \, {\rm d}s \right)^{\frac{1}{2}},
	\end{align}	
	where $0<M_0:=\|u_{0,x}\|_2 + \int_0^\infty \|f_x(s,x)\|_2\, {\rm d}s <+\infty.$
\end{theorem}

\begin{proof} We choose a cut-off function $\eta(s)\in C^{\infty}(\mathbb{R})$ in time such that
	$0\leq \eta\leq1,\eta(s)=1$ for $s\geq\frac{t}{2}$, $\eta(s)=0$ for $s\leq\frac{t}{4}$ and $\eta'(s)\leq\frac{C}{t}$ for some constant $C>0.$
	Now we choose $\eta u_{xx}$ as a test-function and obtain
	\begin{align}\label{add01}
	\int_{0}^{t}\int_{\Omega}\eta(s)u_{s}u_{xx}{\rm d} x{\rm d} s-\int_{0}^{t} \int_{\Omega} \eta(s)u^{\alpha + 2} |u_{xxx}|^{\alpha+1}{\rm d} x{\rm d} s
	=-\int_{0}^{t}\int_{\Omega} \eta(s)f_x(s,x)u_{x}{\rm d} x{\rm d} s.
	\end{align}
	Notice that $\eta(0)=0,\eta(t)=1$, we have
	\begin{align}\label{add02}
	0\leq E[u](t)=\int_{0}^{t}\frac{{\rm d} }{{\rm d} s}(\eta(s)E[u](s)){\rm d} s
	=\int_{0}^{t}\eta'(s)E[u](s){\rm d} s-\int_{0}^{t}\int_{\Omega}\eta(s)u_{s}u_{xx}{\rm d} x{\rm d} s.
	\end{align}		
	Combining \eqref{add01} with \eqref{add02}, we deduce
	\begin{align}\label{add03}
	\int_{\frac{t}{2}}^{t}D[u](s){\rm d} s&\leq \int_{0}^{t}\eta(s)D[u](s){\rm d} s
	\nonumber\\ &\leq\int_{0}^{t}\eta'(s)E[u](s){\rm d} s+\int_{0}^{t}\int_{\Omega} \eta(s)f_x(s,x)u_{x}{\rm d} x{\rm d} s
	\nonumber\\ 
	&\leq \frac{C}{t}\int_{\frac{t}{4}}^{\frac{t}{2}}E[u](s){\rm d} s+C\left( \int_{\frac{t}{4}}^{t} \|f_x(s,x)\|_2^2 {\rm d} s\right)^{\frac{1}{2}}\left( \int_{\frac{t}{4}}^{t}E[u](s) {\rm d} s\right)^{\frac{1}{2}}.
	\end{align}
	Furthermore, substituting \eqref {E.1} and \eqref {E.2} into \eqref {add03}, we complete the proof of Theorem \ref {D}.
\end{proof}

\begin{remark}
	{\rm Through analogous computations as in Remark \ref{u000}, we also observe that 
		 the significant influence of the forcing terms on the decay rate of the dissipation functional in the $L^1$-in-time sense.
	}
\end{remark}

\section{Constant forces}

In this section, the forcing term is a fixed time-and space-independent force $f_0$, we
investigate the following problem:
\begin{equation}\label{1.9}
\begin{cases}
u_t + \left(u^{\alpha+2} |u_{xxx}|^{\alpha-1}u_{xxx} \right)_x = f_0,~~&(t,x) \in(0,\infty) \times\Omega,\\
u_x(t,x) =u_{xxx} (t,x) = 0,~~&(t,x) \in (0,\infty) \times \partial\Omega,\\
u(0,x)=u_{0}(x),~~&x \in \Omega.
\end{cases}
\end{equation}

The local existence of weak solutions to problem \eqref{1.9} follow analogously to Theorem \ref{sth.local}, and thus we omit the details here.
%
%
Here, we focus on the case $f_0 \ge 0$, since for $f_0 < 0$, the mass equation \eqref{sth.local.mass} yields
$$u(t, x) \rightarrow 0 \text{ as } t \rightarrow T^* := -\frac{\bar{u}_0}{f_0},  $$
indicating that the thin-film will completely dry out over the finite time  $T^*$.

\begin{theorem}\label{f0.long}
	
	Fix $\alpha > 0$. If $0<u_0(x) \in H^1(\Omega)$ and $f_0\ge 0$ satisfy
	
	{\rm (i)}  $\|u_{0,x}\|_2<\left| \Omega\right|^{-\frac{1}{2}}\bar{u}_0$, \\
	or			
	{\rm (ii)} $\left| \Omega\right|^{-\frac{1}{2}}\bar{u}_0\le \|u_{0,x}\|_2<\left| \Omega\right|^{-\frac{1}{2}}(u_0(x)+\bar{u}_0-\varepsilon)$ with $\varepsilon \in (\bar{u}_0-\left| \Omega\right|^{\frac{1}{2}} \|u_{0,x}\|_2,   u_0(x)+\bar{u}_0-\left| \Omega\right|^{\frac{1}{2}} \|u_{0,x}\|_2)$, and there exists $\eta > 0$ such that $f_0\ge \eta.$\\
	Then problem \eqref{1.9} possesses at least one global positive weak solution $u$ that satisfies the  following asymptotic behavior:	
	
	{\rm (1)}  In the shear-thickening case $0 < \alpha < 1$,  there exists a positive but finite time $0 < t^* < \infty$ such that
	\begin{align*}
		u(t,x) \rightarrow \bar{u}_0 +t^*f_0 \text{ ~in~ } H^1(\Omega), \text{~as } t \rightarrow t^*,
		\text{ ~and~ } u(t,x) = \bar{u}_0 +tf_0 , \quad t \geq t^*, \, x \in \Omega.
	\end{align*}
	
	{\rm (2)} In the shear-thinning case $1 < \alpha < \infty$, there exists a constant $C > 0$ such that
\begin{align*}
	\| u - \bar{u}_0 -tf_0 \|_{H^1(\Omega)}
	\le \frac{C\|u_{0,x}\|_2}
	{\left[1	+C\|u_{0,x}\|_2^{\alpha-1}h^{\alpha+2}t
		\right] ^{\frac{1}{\alpha-1}}}, ~~~\forall t >0.
	\end{align*}

	{\rm (3)} In the Newtonian case $\alpha = 1$, there exist positive constants $C > 0$ such that
	\begin{align*}
	\| u - \bar{u}_0 -tf_0 \|_{H^1(\Omega)}
	\le C \|u_{0,x}\|_2e^{-Ch^{3}t}, ~~~\forall t >0,
	\end{align*}
		where  $m=  \bar{u}_0 - \left| \Omega\right|^{\frac{1}{2}} \|u_{0,x}\|_2$,
		\begin{align*}
		h =\begin{cases}
		m>0,&\text{~if~} \|u_{0,x}\|_2<\left| \Omega\right|^{-\frac{1}{2}}\bar{u}_0,\\
		\frac{-m+\varepsilon}{2}>0,&\text{~if~} \left| \Omega\right|^{-\frac{1}{2}}\bar{u}_0\le \|u_{0,x}\|_2<\left| \Omega\right|^{-\frac{1}{2}}(u_0(x)+\bar{u}_0-\varepsilon).
		\end{cases}
		\end{align*}
\end{theorem}

\begin{remark}
	{\rm  \cite[Theorem 1.1]{JJCLKN} gave the convergence of solutions for $f_0=0$. 
		Under low initial energy (i), we extend this result to any $f_0\geq 0$, obtaining the long-time behavior of solutions in  $H^1(\Omega)$. 
		Under high initial energy (ii), where the conditions of \cite{JJCLKN} fail, we still obtain the desired estimates by exploiting the compensating effect of $f_0$ on the positivity of the solution $u$.}
\end{remark}

\begin{remark}
	{\rm  Theorem \ref{f0.long} provides the quantitative blow-up rate for weak solutions to problem \eqref{1.9}.  
		Specifically,    $u$ approaches $\bar{u}_0 + tf_0$ in $H^1(\Omega)$ at at a polynomial rate as $t\to\infty$ for the  shear-thinning case $1 < \alpha < \infty$ and exponential rate for the Newtonian case $\alpha = 1$.
Particularly, for $0 < \alpha < 1$, $u$ coincides exactly with $\bar{u}_0 + tf_0$ within a finite time $0<t^*<\infty$.}
\end{remark}

{\bf Proof of Theorem \ref{f0.long}.}
Let
$$w(t, x) := u(t, x) - u_{\Omega}(t), $$
where
\begin{align}\label{uOmega}
	u_{\Omega}(t) := \frac{1}{|\Omega|} \int_\Omega u(t,x)\,{\rm d}x= \bar{u}_0  + tf_0.
\end{align}
Note that  $w(t, x)$  solves
\begin{equation}\label{con}
\begin{cases}
w_t + \left( u^{\alpha+2} |w_{xxx}|^{\alpha-1} w_{xxx}\right) _x = 0,~~&(t,x) \in (0,\infty) \times \Omega,\\
w_x =  w_{xxx} = 0,~~&(t,x) \in (0,\infty) \times \partial\Omega,\\
w(0,x)=w_{0}(x):=u_{0}(x)-\bar{u}_0,~~&x \in \Omega.
\end{cases}
\end{equation}
From \eqref{sth.local.mass}, it is   easy to know
\begin{align}\label{w}
\int_{\Omega} w(t, x) \,{\rm d}x = \int_{\Omega} w_0(x) \,{\rm d}x = 0.
\end{align}
According to \eqref{sth.local.energy}, we obtain
\begin{align}\label{con.equ}
\frac{1}{2} \frac{\rm d}{{\rm d}t} \int_{\Omega} \left| w_x\right| ^2  {\rm d}x + \int_{\Omega}u^{\alpha+2} |w_{xxx}|^{\alpha+1}   \, {\rm d}x = 0,~~~\forall t\in [0,T].
\end{align}
The non-negativity of the second term on the left-hand side of equation \eqref{con.equ} yields
$$\| w_x \|_2  \le \| w_{0,x} \|_2,~~~\forall t\in [0,T].  $$
Then by \eqref{uOmega}, the fundamental theorem of calculus and  Poincar\'e's  inequality, we know
\begin{align*}
\left|u- \bar{u}_0  - tf_0\right|
\le |\Omega|^{\frac{1}{2}} \|u_x \|_2
\le |\Omega|^{\frac{1}{2}} \| u_{0,x} \|_2,~~~\forall t\in [0,T],
\end{align*}
which implies
\begin{align}\label{low}
	u(t, x) \ge m +tf_0,~~~\forall t\in[0,T],
\end{align}
where $m=  \bar{u}_0 - \left| \Omega\right|^{\frac{1}{2}} \|u_{0,x}\|_2.$

Next, we claim that  for all $\varepsilon \in (\bar{u}_0-\left| \Omega\right|^{\frac{1}{2}} \|u_{0,x}\|_2,   u_0(x)+\bar{u}_0-\left| \Omega\right|^{\frac{1}{2}} \|u_{0,x}\|_2)$ and  $ t\in[0,T]$, solutions $ u $ to \eqref{1.9} on $ [0, T] $ remain strictly bounded away from zero:
\begin{align}\label{lower}
	u(t,x)&\ge  h
	  =\begin{cases}
	m>0,&\text{~if~} \|u_{0,x}\|_2<\left| \Omega\right|^{-\frac{1}{2}}\bar{u}_0 ,\\
	\frac{-m+\varepsilon}{2}>0,&\text{~if~} \left| \Omega\right|^{-\frac{1}{2}}\bar{u}_0\le \|u_{0,x}\|_2<\left| \Omega\right|^{-\frac{1}{2}}(u_0(x)+\bar{u}_0-\varepsilon).
	\end{cases}
\end{align}
Indeed, if $\|u_{0,x}\|_2<\left| \Omega\right|^{-\frac{1}{2}}\bar{u}_0,$  then  \eqref{lower}$_1$ follows directly from \eqref{low}.

If  $\left| \Omega\right|^{-\frac{1}{2}}\bar{u}_0\le \|u_{0,x}\|_2<\left| \Omega\right|^{-\frac{1}{2}}(u_0(x)+\bar{u}_0-\varepsilon)$, we have
$$u_0(x)>-m+\varepsilon>0.$$
We define the time
$$  \tau = \sup \{\tilde{T} \geq 0 ~|~ \exists \text{ a weak solution } u \text{ to } \eqref{1.9}  \text{ with }  u(t, x) \ge \frac{-m+\varepsilon}{2}, \, \forall\, 0 < t < \tilde{T} \},  $$
which is the maximum time up to which solution remains bounded away from zero. Note that by the continuity of weak solutions, we have $0 < \tau \leq T$ and
\begin{align}\label{low.1}
	u(t, x) \ge \frac{-m+\varepsilon}{2},~~~\forall t\in[0,\tau].
\end{align}
Moreover, from \eqref{low} we have
 \begin{align}\label{low.2}
 u(t, x) \ge m +tf_0\ge \frac{-m+\varepsilon}{2},~~~\forall t\in[T_0,T],
 \end{align}
 where $ T_0=\frac{-3m+\varepsilon}{2f_0}.$
Then we may choose $\eta > 0$ large enough so that we obtain $T_0 \le\tau$.
Combining with \eqref{low.1} and \eqref{low.2}, we can get \eqref{lower}.
Therefore,  by bootstrapping as in Remark \ref{global},
this solution $u$ can be extended globally,  and satisfies
$$u(t,x)\ge  h > 0,~~~\forall t\ge 0,$$
and similar estimates of \eqref{dissipation.2}, \eqref{dissipation.3} and from \eqref{lower}, we find
\begin{align}\label{con.inter}
\int_{\Omega}  u^{\alpha + 2} |w_{xxx}|^{\alpha + 1} \, {\rm d}x
\geq h^{\alpha+2} \|w_{xxx}\|^{\alpha + 1}_{\alpha + 1}
\geq  C_6 h^{\alpha+2}  \|w_x\|_2^{\alpha + 1},~~~\forall t \ge 0,
\end{align}
where $C_6:= |\Omega|^{-\frac{5\alpha + 3}{2}} $.
And then substituting \eqref{con.inter} into \eqref{con.equ} yields
\begin{align}\label{con.ineq}
\frac{\rm d}{{\rm d}t} \|w_x\|_2^2
+ 2C_6  h^{\alpha+2} \|w_x\|_2^{\alpha + 1}\le 0,~~~\forall t\ge 0.
\end{align}

{\bf Case 1.}  If $0<\alpha<1$.
Integrating \eqref{con.ineq} from $0$ to $t$ and obtain
\begin{align*}
	\|w_x\|_2^2
	\le \left[\|w_{0,x}\|_2^{1-\alpha}
	-(1-\alpha)C_6 h^{\alpha+2}t\right] ^{\frac{2}{1-\alpha}}, \text{~~if~} \|w_x\|^2_2>0,
\end{align*}
which implies the existence of a finite time $0\le t^*<\infty$ with
$$t^*\le \frac{\|w_{0,x}\|_2^{1-\alpha}}{(1-\alpha)C_6h^{\alpha+2}}$$
satisfying
$$\|w_x\|_2^2=0,~~~\forall t \ge t^*. $$
Note that $\|w_x\|_2^2=0$  for all $t \geq t^*$
 and \eqref{w} implies that $w(t, x) = 0$ for all $t \geq t^*$ and $x \in \Omega$.
Thus, we obtain that
\begin{align*}
u(t,x) \rightarrow \bar{u}_0 +t^*f_0 \text{ ~in~ } H^1(\Omega), \text{~as } t \rightarrow t^*,
\text{ ~and~ } u(t,x) = \bar{u}_0 +tf_0 , \quad t \geq t^*, \, x \in \Omega.
\end{align*}

{\bf Case 2.}  If $\alpha>1$, then we have
\begin{align}\label{xxx.1}
\|w_x\|_2^2
\le \left[\|w_{0,x}\|_2^{1-\alpha}
+(\alpha-1)C_6h^{\alpha+2}t
\right] ^{\frac{2}{1-\alpha}},  ~~~\forall t >0.
\end{align}
By \eqref{xxx.1} and Poincar\'e's  inequality, we deduce
\begin{align*}
\| u-\bar{u}_0 -tf_0 \|_{H^1(\Omega)}^2
\le \frac{C_4\|u_{0,x}\|_2^2}
{\left[1
	+ (\alpha-1)C_6\|u_{0,x}\|_2^{\alpha-1}h^{\alpha+2}t
	\right] ^{\frac{2}{\alpha-1}}},~~~\forall t >0,
\end{align*}
where $C_4=\left(\frac{|\Omega|}{\pi} \right) ^2+1$.

{\bf Case 3.}  If $\alpha=1$, then we get
\begin{align}\label{xxx.2}
	 \|w_x\|_2^2
	\le \|w_{0,x}\|_2^2e^{-2C_6h^{\alpha+2}t},  ~~~\forall t >0.
\end{align}
According to \eqref{xxx.2} and Poincar\'e's  inequality, we obtain
\begin{align*}
\|u-\bar{u}_0 -tf_0\|_{H^1(\Omega)}^2
\le C_4 \|u_{0,x}\|_2^2e^{-2C_6h^{3}t},~~~\forall t >0.
\end{align*}
Therefore, the proof of the Theorem \ref{f0.long} is complete.
\(\square\)


\section{Numerical Simulations}

This section verifies the long-time behavior of the solution to problem \eqref{1.1} and its approximation properties to  $\bar{u}_0 + \frac{1}{|\Omega|}\int_0^\infty \int_{\Omega} f(s,x){\rm d}x{\rm d}s$, with a focus on the decay of their error in the $H^1(\Omega)$-norm.
Employing a finite difference discretization scheme, we implement numerical experiments in Python. 
The results not only clearly simulate the asymptotic stability of the solution and the dynamic process of the decay of the error norm,  but also demonstrate the theoretical prediction of Theorem \ref{sth.long}.

\begin{example}\label{ex1}
	{\rm Let us consider the numerical solution to problem \eqref{1.1} on
		$$\alpha = 0.5,~\Omega=[0,200],~T = 20,$$
		under the initial data
	$$u_0(x)=3 + 0.01\cos(\frac{\pi x}{10}),$$
	 and  the time-and space-dependent force
	 	$$f_1(t,x)=e^{-t}(1+0.01\cos(\frac{\pi x}{10})) \text{~~or~~}f_2(t,x)=(1+t)^{-2}(1+0.01\cos(\frac{\pi x}{10})).$$
According to Remark \ref{u0}, $u_0 (x) $ and  $f_i(t, x)~(i=1,2)$ satisfy the sufficient conditions in Theorem \ref{sth.long}. 
Figure \ref{fig33} and Figure \ref{fig44} depict the time evolution of the numerical solution driven by the forcing terms $f_1(t,x)$ and $f_2(t,x)$, which decay exponentially and polynomially in time, respectively. 
The results show that the solution gradually converges to the positive constant 4.   
From a physical perspective, the results also indicate that the film surface will gradually flatten over time.

\begin{figure}[htbp]
	\centering
	\begin{minipage}[t]{0.48\textwidth}
		\centering
		\includegraphics[width=7cm]{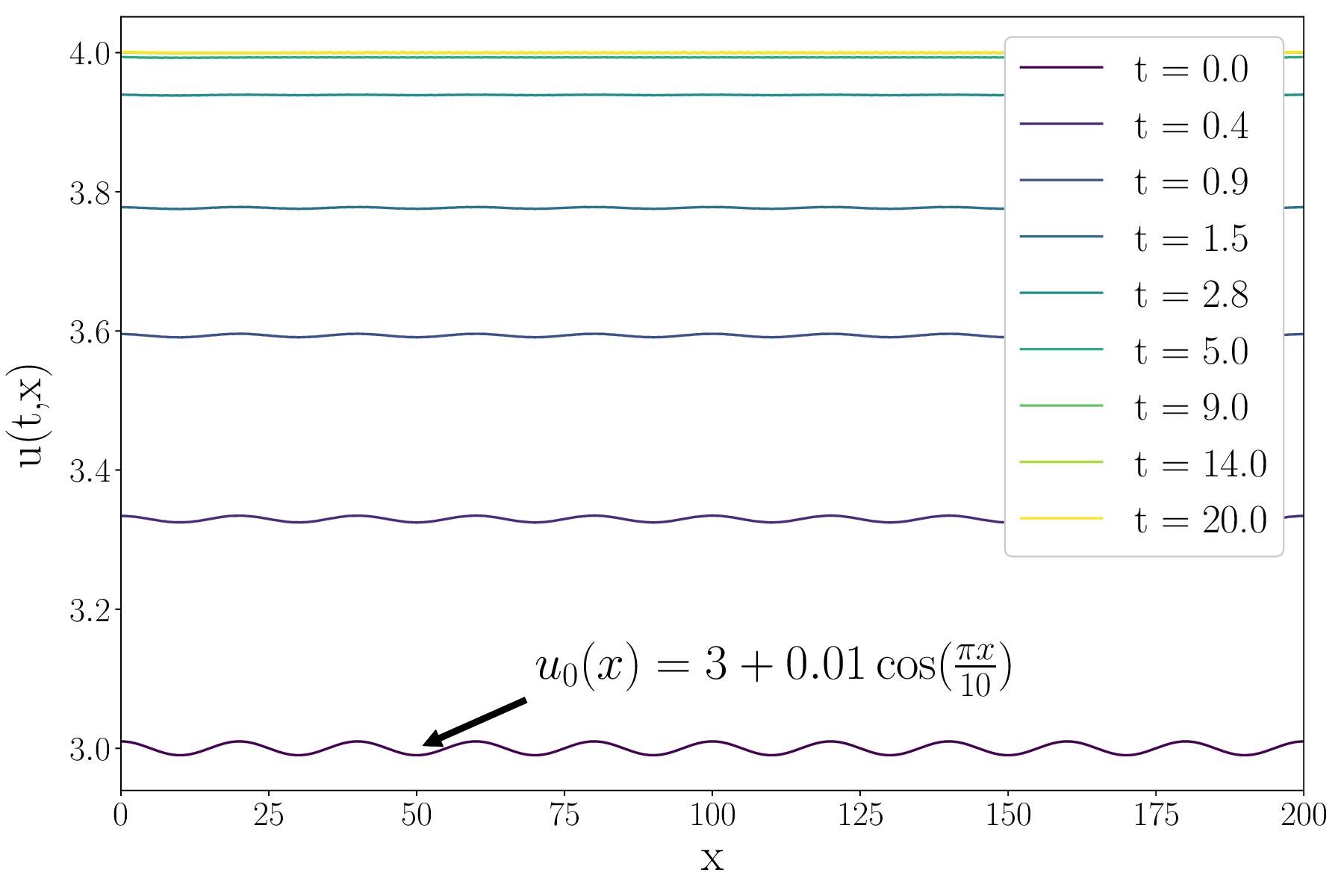}
		\caption{Film thickness evolution}
		\label{fig33}
	\end{minipage}
	\begin{minipage}[t]{0.48\textwidth}
		\centering
		\includegraphics[width=7cm]{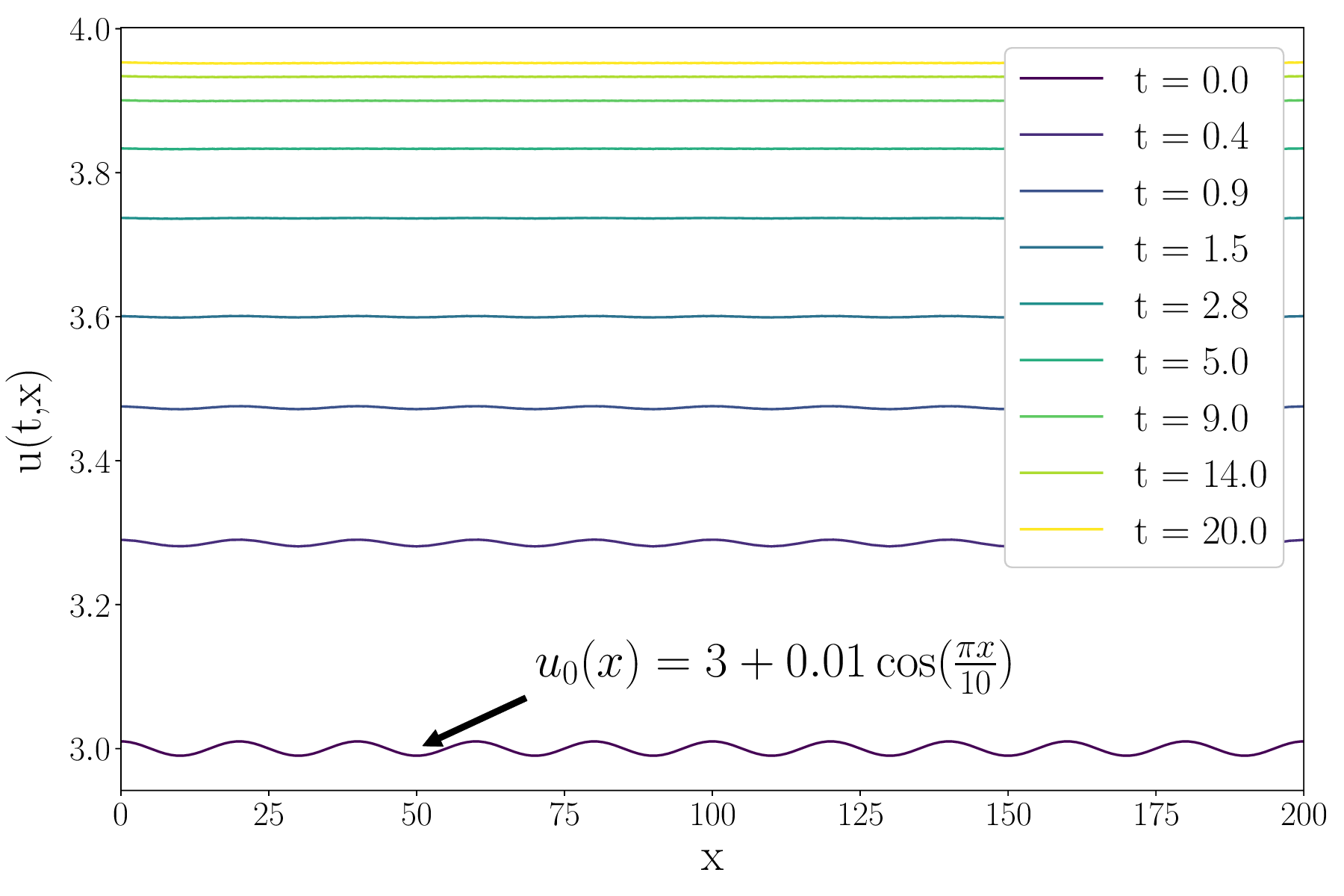}
		\caption{ Film thickness evolution }
		\label{fig44}
	\end{minipage}
\end{figure}

}
\end{example}


Further, figure \ref{fig32} presents the evolution of the error norm  $\| u - \bar{u} _0 - \frac{1}{|\Omega |} \int_0^\infty \int_{\Omega} f_i(s,x){\rm d}x{\rm d}s\|_{H^1{(\Omega)}}$ over time by computing its numerical approximation at each time step, clearly reflecting its decay behavior. 
It can be observed that for the exponentially decaying forcing term $f_1(t,x)$, the corresponding error norm decays significantly faster than that for the polynomially decaying force $f_2(t,x)$, which is completely consistent with the theoretical prediction of Theorem \ref{sth.long}.

\begin{figure}[htbp]
	\centering
	\includegraphics[width=8cm]{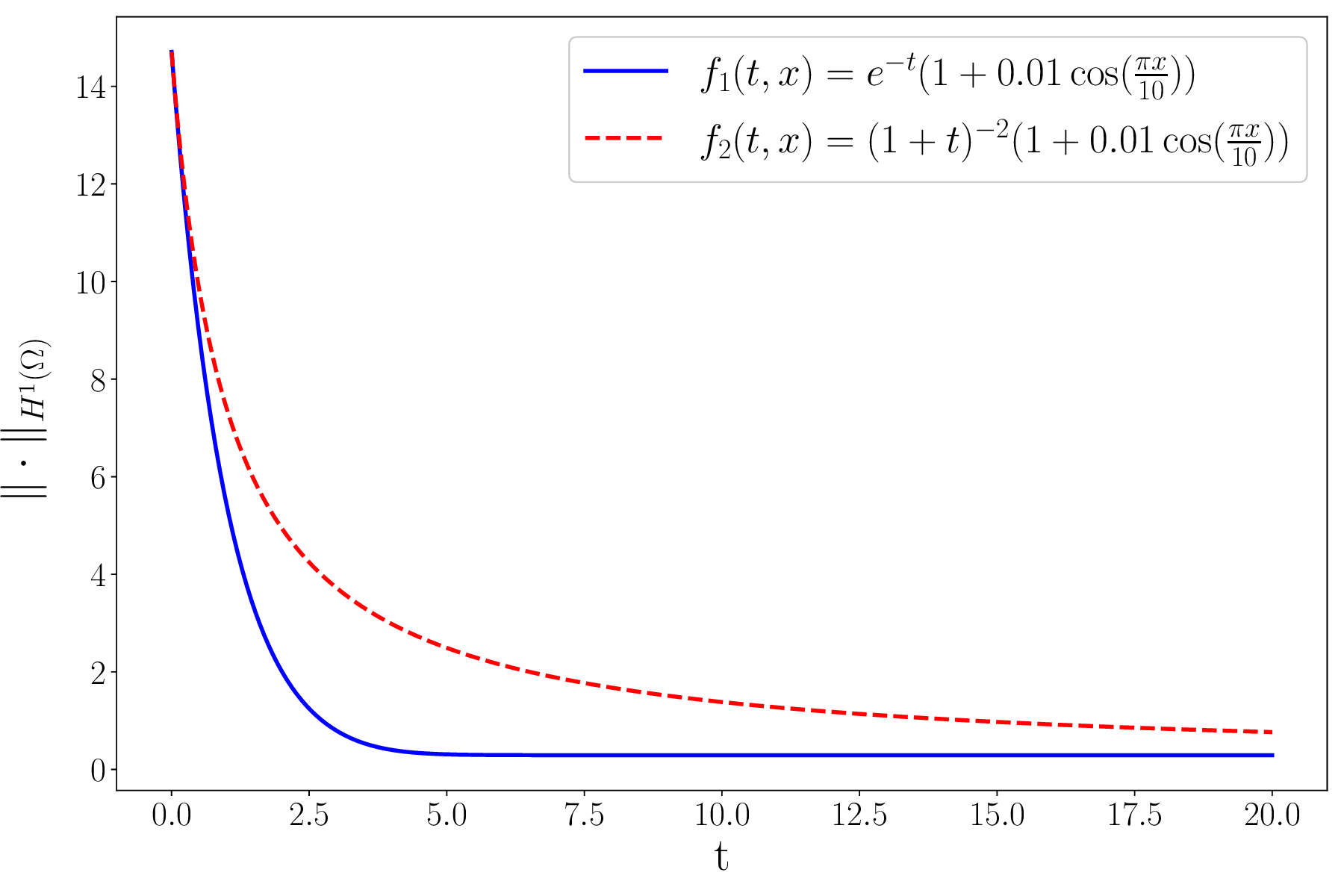} 
	\caption{Plot of $H^1(\Omega)$-norm over time $t$}
	\label{fig32}
\end{figure}




\section{Comments and further problems}

The core contribution of this paper is the development of a novel and versatile framework for the asymptotic analysis of non-autonomous thin-film equations, which is also applicable to other types of non-Newtonian thin-film equations, such as the Ellis law thin-film equation. 
The proposed "two-sided estimates" overcomes the essential difficulties caused by the lack of mass conservation and energy monotonicity, clearly reveals the regulatory mechanism of forcing terms and dissipation terms on the asymptotic behavior of solutions, and provides explicit two‑sided estimates for the convergence rate. 
These results not only extend existing theories for the homogeneous case but also lays a solid mathematical foundation for the quantitative analysis of related physical models. 

~

{\bf Further problems.} Several important questions naturally arise from this study and merit deeper exploration:

$\bullet$ It is natural to ask whether the technical condition imposed on the forcing terms in Theorem \ref{sth.long} can be weakened without compromising the characterization of the long-time behavior.

$\bullet$ What is the long-time behavior under more complex external forces, such as periodic, impulsive, or stochastic forcing?

$\bullet$ Is it possible to extend the analytical framework of this paper to more general nonlinear diffusion terms or nonlocal models? 

$\bullet$ How to investigate the long-time behavior of solutions in higher dimensions or on domains with complex geometry, such as those involving free boundaries?


~


\noindent {\bf Author Contributions} J.H Zhao and B. Guo wrote the main manuscript text and revised the manuscript. All authors reviewed the manuscript.

~

\noindent {\bf Funding} This paper has been partially supported by  the National Natural Science Foundation of China (NSFC) (No. 11301211) and the Natural Science Foundation of Jilin Province, China (No. 201500520056JH).

~

\noindent {\bf Data Availability} No datasets were generated or analyzed during the current study.

\end{document}